\documentclass[11pt]{article}
\usepackage{amsmath}
\usepackage{amsfonts}
\usepackage{graphicx}
\usepackage{amssymb}
\usepackage{leftidx}

\usepackage[notref,notcite]{showkeys}

\newtheorem{condition**}{A*}
\newtheorem{condition***}{C*}
\newtheorem{condition*}{C}

\newtheorem{proposition}{Propsition}[section]

\newtheorem{definition}{Definition}[section]
\newtheorem{theorem}{Theorem}[section]
\newtheorem{lemma}{Lemma}[section]
\newtheorem{remark}{Remark}[section]

\newenvironment{keywords}{{\bf Key words: }}{}

\begin{document}

\title{Linear Quadratic Mean Field Game with Control Input Constraint}

\author{Ying Hu$\;^{a}$, Jianhui Huang$\;^{b}$, Xun Li$\;^{c,\ast}$\bigskip\\{\small ~$^{a}$IRMAR, Universit\'e Rennes 1, Campus de Beaulieu, 35042 Rennes Cedex, France}\\{\small $^{b}$Department of Applied Mathematics, The Hong Kong Polytechnic University, Hong Kong}\\{\small ~$^{c}$Department of Applied Mathematics, The Hong Kong Polytechnic University, Hong Kong}}

\maketitle

\begin{abstract}
In this paper, we study a class of linear-quadratic (LQ) mean-field games in which the individual
control process is constrained in a closed convex subset $\Gamma$ of full space $\mathbb{R}^{m}$.
The decentralized strategies and consistency condition are represented by a class of mean-field forward-backward stochastic differential equation (MF-FBSDE)
with projection operators on $\Gamma$. The wellposedness of consistency condition system is obtained using the monotonicity condition method. The related $\epsilon$-Nash equilibrium property is also verified.
\end{abstract}


\footnotetext[1]{{\scriptsize Corresponding author: Ying Hu (Ying.Hu@univ-rennes1.fr)}}
\renewcommand{\thefootnote}{\arabic{footnote}}
\footnotetext[2]{{\scriptsize The work of Ying Hu is supported  by Lebesgue center of mathematics ``Investissements d'avenir"
program - ANR-11-LABX-0020-01, by ANR-15-CE05-0024-02 and by ANR-MFG
; The work of James Jianhui Huang is supported by PolyU G-YL04, RGC Grant 502412,
15300514; The work of Xun Li is supported by PolyU G-UA4N, Hong Kong RGC under grants 15224215 and 15255416.}}
\renewcommand{\thefootnote}{\fnsymbol{footnote}}
\footnotetext{\textit{{\scriptsize E-mail addresses:}}
{\scriptsize Ying.Hu@univ-rennes1.fr (Ying\ Hu); majhuang@polyu.edu.hk (Jianhui\ Huang); malixun@polyu.edu.hk (Xun \ Li).}}

\medskip

\section{Introduction}

Our starting point comes from the recently well-studied mean-field games (MFGs) for large-population system. The large-population system arises naturally in various fields such as economics,
engineering, social science and operational research, etc. The most salient feature of large-population system is
the existence of a large number of individually negligible agents (or players) which are interrelated in their dynamics and (or) cost functionals
via the state-average (in linear case) or more generally, the generated empirical measure over the whole population (in nonlinear case). Because of this highly complicated coupling feature,
it is intractable for a given agent to employ the centralized optimization strategies based on the information of all its peers in large-population system. Actually, this will bring considerably high computational complexity in a large-scale manner. Alternatively, one reasonable and practical direction is to investigate the related decentralized strategies based on local information only.
By local information, we mean that the related strategies should be designed upon the individual state (or, random noise) of given agent together with some mass-effect quantities which can be computed in off-line manner.

Along this research direction, one efficient and tractable methodology leading to decentralized strategies is the MFGs which generally
lead to a coupled system of HJB equation and Fokker-Planck (FP) equation in nonlinear case. In principle,
the procedure of MFGs consists of the following four main steps (see \cite{ll}, \cite{BFY}, \cite{hcm06}, \cite{H-C-M 2007}, \cite{CD}, \cite{wz}, etc):
in Step 1, it is necessary to analyze the asymptotic behavior of state-average when the agent number $N$ tends to
infinity and introduce the related state-average limiting term. Of course, this limiting term is undetermined at this moment,
thus it should be treated as some exogenous ``frozen" term; Step 2 turns to study the related limiting optimization problem
(which is also called auxiliary or tracking problem) by adopting the state-average instead of its frozen limit term.
The initial high-coupled optimization problems of all agents are thus decoupled and only parameterized by this generic frozen limit.
The related decentralized optimal strategy can be analyzed using standard control techniques such as dynamic programming principle (DPP)
or stochastic maximum principle (SMP) (see e.g., \cite{yz}). As a result, some HJB equation (due to DPP) or Hamiltonian system (due to SMP) will be obtained to characterize this decentralized optimality; Step 3 aims to determine the frozen state-average limit by some consistency condition: when applying the optimal decentralized strategies derived in Step 2,
the state-average limit should be reproduced as the agents number tends to infinity. Accordingly,
some fixed-point analysis should be applied here and some FP equation will be introduced and coupled with the HJB equation in Step 2.
As the necessary verification, Step 4 will show that the derived decentralized strategies should possess the $\epsilon$-Nash equilibrium properties. A comprehensive survey of MFG can be found in
\cite{pc}.

For further analysis of MFGs, the interested readers may refer to \cite{G-L-L 2011} for a
survey of mean-field games focusing on the partial differential equation aspect and related real applications;
\cite{BFY} for more recent MFG studies and the related mean-field type control; \cite{CD} for the probabilistic analysis of a large class of
stochastic differential games for which the interaction between the players is of mean-field type; \cite{cfs} for
the mean-field game where considerable interrelated banks share the system risk and common noise;
\cite{T-Z-B 2014} for a class of risk-sensitive mean-field stochastic differential games;
\cite{hmc06} for MFGs with nonlinear diffusion dynamics and their relations to McKean-Vlasov particle system.
It is remarkable that there exists a substantial literature body to the study of MFGs in linear-quadratic (LQ) framework.
Here, we mention a few of them which are more relevant to our current work: \cite{H 2010} the mean-field LQ games with a major player and a large number of minor players, \cite{H-C-M 2007}
the mean-field LQ games with non-uniform agents through the state-aggregation by empirical distribution, \cite{N-H 2012}
the mean-field LQ mixed games with continuum-parameterized minor players.

In this paper, we discuss the linear-quadratic (LQ) mean-field game where the individual control is constrained in a closed convex set $\Gamma$ of full space: $\Gamma \subset \mathbb{R}^m$. The LQ problems with control constraint arise naturally from various practical applications. For instance, the no-shorting constraint in portfolio selection leads to the LQ control with positive control ($\Gamma=\mathbb{R}^{m}_{+},$ the positive orthant). Moreover, due to general market accessibility constraint, it is also interesting to study the LQ control with more general closed convex cone constraint (see \cite{hz}). As a response, this paper investigates the LQ dynamic game of large-population system with general closed convex control constraint. Our investigation is mainly sketched as follows. First, applying the maximum principle, the optimal decentralized response is characterized through some Hamiltonian system with projection operator upon the constrained set $\Gamma$. Second, the consistency condition system is connected to the well-posedness of some mean-field forward-backward stochastic differential equation (MF-FBSDE). Next, we present some monotonicity condition of this MF-FBSDE to obtain its uniqueness and existence. Last, the related approximate Nash equilibrium property is also verified. We derive the MFG strategy in its open-loop manner. Consequently, the approximate Nash equilibrium property is verified under the open-loop strategies perturbation and some estimates of forward-backward SDE are involved. In addition, all agents are set to be statistically identical thus the limiting control problem and fixed-point arguments are given for a \emph{representative} agent. In case the agents are heterogeneous with different parameters, the similar procedure to MFG strategies can be proceeded via the introduction of index indicator and empirical state-average statistics. 

The reminder of this paper is structured as follows: Section 2 formulates the LQ MFGs with control constraint.
The decentralized strategies are derived with the help of a forward-backward SDE with projection operators.
The consistency condition is also established. Section 3 verifies the $\epsilon-$Nash equilibrium of the decentralized strategies.
Section 4 is appendix.


\section{Mean-Field LQG Games with Control Constraint}

Throughout this paper, we denote the $k$-dimensional Euclidean space by $\mathbb{R}^{k}$ with standard Euclidean norm $|\cdot|$ and standard Euclidean inner product $\langle\cdot,\cdot\rangle$.
The transpose of a vector (or matrix) $x$ is denoted by $x^{T}$. $\textrm{Tr}(A)$ denotes the trace of a square matrix $A$. Let $\mathbb{R}^{m \times n}$ be the Hilbert space consisting of all ($m\times n$)-matrices with the inner product $\langle A,B\rangle:=\textrm{Tr}(AB^{T})$ and the norm $|A|:=\langle A,A\rangle^{\frac{1}{2}}$.
Denote the set of symmetric $k \times k$ matrices with real elements by $S^k$.
If $M \in S^k$ is positive (semi)definite,
we write $M>\ (\geq)\ 0$. $L^{\infty}(0, T; \mathbb{R}^{k})$ is the space of uniformly bounded $\mathbb{R}^{k}-$valued functions.
If $M(\cdot)\in L^{\infty}(0, T; S^k)$ and $M(t)>\ (\geq)\ 0$ for all $t\in[0,T]$,
we say that $M(\cdot)$ is positive (semi) definite, which is denoted by $M(\cdot)>\ (\geq)\ 0$. $L^{2}(0, T; \mathbb{R}^{k})$
is the space of all $\mathbb{R}^{k}-$valued functions satisfying $\int_0^{T}|x(t)|^{2}dt<\infty.$ \\

Consider a finite time horizon $[0,T]$ for fixed $T>0$. We assume $(\Omega,
\mathcal F, \{\mathcal{F}_t\}_{0 \leq t \leq T}, P)$ is a complete, filtered probability space on which a
standard $N$-dimensional Brownian motion $\{W_i(t),\ 1\le i\leq N\}_{0 \leq t \leq T}$ is defined. For given filtration $\mathbb{F}=\{\mathcal{F}_t\}_{0 \leq t \leq T},$
let $L^{2}_{\mathbb{F}}(0, T; \mathbb{R}^{k})$ denote the
space of all $\mathcal{F}_t$-progressively measurable $\mathbb{R}^{k}$-valued processes satisfying $\mathbb{E}\int_0^{T}|x(t)|^{2}dt<\infty.$
Let $L^{2, \mathcal{E}_0}_{\mathbb{F}}(0, T; \mathbb{R}^{k}) \subset L^{2}_{\mathbb{F}}(0, T; \mathbb{R}^{k})$ the subspace satisfying $\mathbb{E}x_t \equiv 0$ for
$x_{\cdot} \in L^{2, \mathcal{E}_0}_{\mathbb{F}}(0, T; \mathbb{R}^{k}).$

Now let us consider a large-population system with $N$ weakly-coupled negligible agents $\{\mathcal{A}_{i}\}_{1 \leq i \leq N}$. The state $x^i$
for each $\mathcal{A}_{i}$ satisfies the following controlled linear stochastic system:
\begin{equation}\label{o1}
\left\{\begin{aligned}
dx^i(t)= & \, [A(t)x^i(t)+B(t)u_i(t)+F(t) x^{(N)}(t)+b(t)]dt \\
         & +[D(t)u_i(t)+\sigma(t)]dW_i(t), \\
\quad x^i(0)= & \, x \in \mathbb{R}^{n},
\end{aligned}\right.
\end{equation}
where $x^{(N)}(\cdot)=\frac{1}{N}\sum_{i=1}^{N}x^{i}(\cdot)$ is the state-average, $(A(\cdot), B(\cdot), F(\cdot), b(\cdot);  D(\cdot), \sigma(\cdot))$
 are matrix-valued functions with appropriate dimensions to be identify soon. For sake of presentation, we set all agents are homogeneous or statistically symmetric
 with same coefficients $(A, B, F, b; D, \sigma)$ and deterministic initial states $x$.

Now we identify the information structure of large population system: $\mathbb{F}^{i}=\{\mathcal F^{i}_t\}_{0 \leq t \leq T}$
 is the natural filtration generated by $\{W_i(t), 0 \leq t \leq T\}$ and augmented by all $P-$null sets in
 $\mathcal F.$ $\mathbb{F}=\{\mathcal F_t\}_{0 \leq t \leq T}$ is the natural filtration generated by $\{W_i(t), 1 \leq i \leq N, 0 \leq t \leq T\}$
 and augmented by all $P-$null sets in $\mathcal F.$ Thus, $\mathbb{F}^{i}$ is the individual decentralized information of $i^{th}$ Brownian motion while $\mathbb{F}$
 is the centralized information driven by all Brownian motion components. Note that the heterogeneous noise $W_i$ is
 specific for individual agent $\mathcal{A}_{i}$ but $x^i(t)$ is adapted to $\mathcal{F}_t$ instead of $\mathcal{F}^{i}_t$
 due to the coupling state-average $x^{(N)}.$

The admissible control $u_i\in \mathcal{U}^{c}_{ad}$ where the admissible control set $\mathcal{U}^{c}_{ad}$ is defined as
$$\mathcal{U}^{c}_{ad}:=\{u_i(\cdot)|u_i(\cdot)\in L^{2}_{\mathbb{F}}(0, T; \Gamma)\},\ 1 \leq i \leq N,
$$where $\Gamma \subset \mathbb{R}^{m}$ is a closed convex set. Typical examples of such set is $\Gamma=\mathbb{R}^{m}_{+}$
which represents the positive control. Moreover, we can also define decentralized control as $u_i\in \mathcal{U}_{ad}^{d,i}$,
where the admissible control set $\mathcal{U}_{ad}^{d,i}$ is defined as
$$
\mathcal{U}_{ad}^{d, i}:=\{u_i(\cdot)|u_i(\cdot)\in L^{2}_{\mathbb{F}^{i}}(0, T; \Gamma)\},\ 1\leq i \leq N.
$$
Note that both $\mathcal{U}_{ad}^{d, i}$ and $\mathcal{U}^{c}_{ad}$ are defined in open-loop sense. Let $u=(u_1, \cdots, u_i, \cdots, u_{N})$ denote the set of control strategies of all $N$ agents and
$u_{-i}=(u_1, \cdots, u_{i-1},$ $u_{i+1}, \cdots, u_{N})$
denote the control strategies set except the $i^{th}$ agent $\mathcal{A}_i.$ Introduce the cost functional of $\mathcal{A}_i$ as
\begin{equation}\label{original cost}
\begin{aligned}
    \mathcal {J}_i(u_i,u_{-i})=&\frac{1}{2}\mathbb{E}\bigg [ \int_0^T \Big <Q(t)\big(x^i(t)-x^{(N)}(t)\big),x^i(t)-x^{(N)}(t)\Big>\\
    &+\big<R(t)u_i(t),u_i(t)\big>dt+\Big<G\big(x^i(T)-x^{(N)}(T)\big),x^i(T)-x^{(N)}(T)\Big>\bigg].
\end{aligned}
\end{equation}
We impose the following assumptions:
\begin{description}
  \item[(H1)] $A(\cdot),  F(\cdot) \in L^\infty(0,T;S^{n}), B(\cdot), D(\cdot) \in L^\infty(0,T;\mathbb{R}^{n\times m}), b(\cdot),
  \sigma(\cdot) \in L^\infty(0,T;\mathbb R^n);$
    \item[(H2)] $Q(\cdot)\in L^\infty(0,T;S^n),Q(\cdot)\geq0, R(\cdot)\in L^\infty(0,T;S^m),R(\cdot)>0$
    and $R^{-1}(\cdot)\in L^\infty(0,T;S^m)$, $G\in S^n,G\geq0$.
\end{description}
It follows that $(\ref{o1})$ admits a unique solution $x^i(\cdot) \in L^{2}_{\mathbb{F}}(0, T; \mathbb{R}^{n})$ under
admissible control $u_i$ with (\textbf{H1}), (\textbf{H2}). Now, we formulate the large population LQG games with control constraint (\textbf{CC}).\\

\textbf{Problem (CC).}
Find an open-loop Nash equilibrium strategies set $\bar{u}=(\bar{u}_1,\bar{u}_2,\cdots,\bar{u}_N)$ satisfying
$$
\mathcal{J}_i(\bar{u}_i(\cdot),\bar{u}_{-i}(\cdot))=\inf_{u_i(\cdot)\in \mathcal{U}_{ad}^c}\mathcal{J}_i(u_i(\cdot),\bar{u}_{-i}(\cdot))
$$
where $\bar{u}_{-i}$ represents $(\bar{u}_1,\cdots,\bar{u}_{i-1},\bar{u}_{i+1},\cdots, \bar{u}_N)$, the strategies of all agents except $\mathcal{A}_i$.

The study of (\textbf{CC}) is of heavy computational burden due to the highly-complicated coupling structure among these agents. Alternatively, one efficient method to search the approximate Nash equilibrium is the mean-field game theory, which bridges the ``centralized" LQG games to the limiting LQG control problems, as the number of agents tends to infinity. To this end, we need to construct some auxiliary control problem using the frozen state-average limit. Based on it, we can find the decentralized strategies by consistency condition. More details are given below. Introduce the following auxiliary problem for $\mathcal{A}_i:$
\[
\left\{\begin{aligned}
dx^i(t)=& \, [A(t)x^i(t)+B(t)u_i(t)+F(t) z(t)+b(t)]dt \\
&+[D(t)u_i(t) +\sigma(t)]dW_i(t), \\
x^i(0)= & \, x \in \mathbb{R}^{n},
\end{aligned}\right.
\]
and limiting cost functional is given by
\begin{equation}\label{limit cost}
\begin{aligned}
    J_i(u_i)=&\frac{1}{2}\mathbb{E}\bigg [ \int_0^T \Big <Q(t)\big(x^i(t)-z(t)\big),x^i(t)-z(t)\Big>\\
    &+\big<R(t)u_i(t),u_i(t)\big>dt+\Big<G\big(x^i(T)-z(T)\big),x^i(T)-z(T)\Big>\bigg].
\end{aligned}
\end{equation}
Now we formulate the following limiting stochastic optimal control (SOC) problem with control constraint (\textbf{LCC}). \\

\textbf{Problem (LCC).} For the $i^{th}$ agent, $i=1,2,\cdots,N,$ find $u^\ast_i(\cdot)\in \mathcal{U}_{ad}^{d,i}$ satisfying
$$J_i(u^\ast_i(\cdot))=\inf_{u_i(\cdot)\in \mathcal{U}_{ad}^{d,i}}J_i(u_i(\cdot)).$$
Then $u^\ast_i(\cdot)$ is called a decentralized optimal control for Problem (\textbf{LCC}).
Note that the cost functional is strictly convex and coercive thus it admits a unique optimal control ${u}_i^{\ast}$.
Now we apply the maximum principle method to characterize $u_i^{\ast}$ with the optimal state ${x}^{i,\ast}.$
First, introduce the following adjoint process
\begin{equation*}
    \left\{
    \begin{aligned}
      dp^i&=-\big[A^Tp^i-Q(x^{i,\ast}-z)\big]dt+q^idW_i(t),\\
      p^i(T)&=-G\big(x^{i,\ast}(T)-z(T)\big).
    \end{aligned}
    \right.
\end{equation*}
Applying the maximum principle, the Hamiltonian function can be expressed by
\begin{equation}\label{Hamiltonian function}
\begin{aligned}
    H^i&=H^i(t,p^i,q^i,x^i,u_i)=\big<p^i,Ax^i+Bu_i+Fz+b\big>\\
    &+\big<q^i,Du_i+\sigma\big>-\frac{1}{2}\big<Q(x^i-z),x^i-z\big>-\frac{1}{2}\big<Ru_i,u_i\big>.
\end{aligned}
\end{equation}
Since $\Gamma$ is a closed convex set, then maximum principle reads as the following local form
\begin{equation}\label{convex maximum principle}
  \left \langle \frac{\partial H^i}{\partial u_i}(t,p^{i,\ast},q^{i,\ast},x^{i,\ast},u^{i,\ast}),u-u^{i,\ast}\right\rangle\leq 0,
  \quad \text{ for all } u\in\Gamma, \text{ a.e. } t\in[0,T],\ \mathbb{P}-a.s.
\end{equation}
Noticing  (\ref{Hamiltonian function}), then (\ref{convex maximum principle}) yields that
\[
   \left\langle B^{T}p^{i,\ast}+D^{T}q^{i,\ast}-Ru^{i,\ast},u-u^{i,\ast}\right\rangle\leq 0, \text{ for all }
   u\in\Gamma, \text{ a.e. } t\in[0,T],\ \mathbb{P}-a.s.
\]
or equivalently (noticing $R>0$),
\begin{equation}\label{optimal control condition}
   \left\langle R^{\frac{1}{2}}[R^{-1}(B^{T}p^{i,\ast}+D^{T}q^{i,\ast})-u^{i,\ast}],R^{\frac{1}{2}}(u-u^{i,\ast})\right\rangle\leq 0,
   \text{ for all } u\in\Gamma, \text{ a.e. } t\in[0,T], \ \mathbb{P}-a.s.
\end{equation}
If we take the following norm on $\Gamma\subset \mathbb{R}^{m}$ (which is equivalent to its Euclidean norm)
\[
 \|x\|^2_{R}=\left\langle \left\langle x,x\right\rangle\right\rangle:=\left\langle R^{\frac{1}{2}}x,R^{\frac{1}{2}}x\right\rangle,
\]
and by the well-known results of convex analysis, we obtain that (\ref{optimal control condition})
is equivalent to
\begin{equation*}
   u^{i,\ast}(t)=\mathbf{P}_{\Gamma}[R^{-1}(t)(B^{T}(t)p^{i,\ast}(t)+D^{T}(t)q^{i,\ast}(t))],
   \quad\text{ a.e. } t\in[0,T], \ \mathbb{P}-a.s.,
\end{equation*}
where $\mathbf{P}_{\Gamma}(\cdot)$ is the projection mapping from $\mathbb{R}^m$ to its closed convex subset $\Gamma$
under the norm $\|\cdot\|_{R}$. For more details, see Appendix. From now on,
we denote
\[
\varphi(t,p,q):=\mathbf{P}_{\Gamma}[R^{-1}(t)(B^{T}(t)p+D^{T}(t)q)].
\]
Then the related Hamiltonian system becomes
\begin{equation*}
\left\{
    \begin{aligned}
      dx^{i,*}&=\Big[Ax^{i,*}+B\varphi(p^{i,*},q^{i,*})+F z+b\Big]dt+\Big[D\varphi(p^{i,*},q^{i,*})+\sigma\Big]dW_i(t),\\
            dp^{i,*}&=-\big[A^Tp^{i,*}-Q(x^{i,*}-z)\big]dt+q^{i,*}dW_i(t),\\
         x^{i,*}(0)&=x,\quad \quad p^{i,*}(T)=-G\big(x^{i,*}(T)-z(T)\big).
    \end{aligned}
    \right.
\end{equation*}
By the consistency condition, it follows that
\begin{equation}\label{limit average process}
    z(\cdot)=\lim_{N\rightarrow+\infty}\frac{1}{N}\sum_{i=1}^{N}x^{i,*}(\cdot)=\mathbb{E}x^{i,*}(\cdot).
\end{equation}
Thus, by replacing $z$ by $\mathbb{E}x^{i,*}$ in above, we get the following system
\begin{equation*}
\left\{
    \begin{aligned}
      dx^{i,*}&=\Big[ Ax^{i,*}+B\varphi(p^{i,*},q^{i,*})+F\mathbb{E}x^{i,*}+b\Big]dt
      +\Big[D\varphi(p^{i,*},q^{i,*})+\sigma\Big]dW_i(t),\\
       dp^{i,*}&=-\big[A^Tp^{i,*}-Q(x^{i,*}-z)\big]dt+q^{i,*}dW_i(t),\\
         x^{i,*}(0)&=x,\quad \quad p^{i,*}(T)=-G\big(x^{i,*}(T)-z(T)\big).
    \end{aligned}
    \right.
\end{equation*}
We have the following consistency condition system for generic agent (we suppress subscript here):
\begin{equation}\label{cc}
\left\{
    \begin{aligned}
      dx&=\Big[Ax+B\varphi(p,q)+F\mathbb{E}x+b\Big]dt+\Big[D\varphi(p,q)+\sigma\Big]dW_t, \\
      -dp&=\big[A^Tp-Q(x-\mathbb{E}x)\big]dt-qdW_t, \\
      x_0&=x,\quad \quad p_{T}=-G\big(x_{T}-\mathbb{E}x_T\big).
    \end{aligned}
    \right.
\end{equation}
The above system is a nonlinear mean-field forward-backward SDE (MF-FBSDE) with projection operator.
It characterizes the state-average limit $z=\mathbb{E}{x}$ and MFG strategies
$\bar{u}_i=\varphi(p, q)$ for a generic agent in the combined manner. As an important issue, we need to prove the above consistency condition system admits a unique solution.
We first present the following uniqueness and existence result.
\begin{theorem}Under \emph{(\textbf{H1}), (\textbf{H2})}, there exists a unique adapted solution $(x, p, q) \in L_{\mathbb{F}^{W}}^{2}(0, T; \mathbb{R}^{n}) \times L_{\mathbb{F}^{W}}^{2, \mathcal{E}_0}(0, T; \mathbb{R}^{n})\times L_{\mathbb{F}^{W}}^{2}(0, T; \mathbb{R}^{n})$ to system \eqref{cc}.\end{theorem}
\begin{proof} (\textbf{Uniqueness}) Suppose that there exists two solutions: $(x^1,p^1,q^1)$, $(x^2,p^2,q^2)$ and denote$$
\hat{x}=x^1-x^2, \quad \hat{p}=p^1-p^2, \quad \hat{q}=q^1-q^2.$$
Then, we have
\begin{equation}\label{e221}
\left\{
        \begin{aligned}
      d\hat{x}&=\Big[A\hat{x}+B\widehat{\varphi}(\hat{p}, \hat{q})+F\mathbb{E}\hat{x}\Big]dt+D\widehat{\varphi}(\hat{p}, \hat{q})dW_t, \\
     -d\hat{p}&=\big[A^T\hat{p}-Q(\hat{x}-\mathbb{E}\hat{x})\big]dt-\hat{q}dW_t, \\
      \hat{x}_0&=0,\qquad \hat{p}_{T}=-G\big(\hat{x}_T-\mathbb{E}\hat{x}_T\big)
    \end{aligned}
    \right.
\end{equation}
with
\begin{equation*}
\begin{aligned}\widehat{\varphi}(\hat{p}, \hat{q})&:=\varphi(p^{1},q^{1})-\varphi(p^{2},q^{2}):=\mathbf{P}_{\Gamma}[R^{-1}(B^{T}p^{1}+D^{T}q^{1})]
-\mathbf{P}_{\Gamma}[R^{-1}(B^{T}p^{2}+D^{T}q^{2})].
\end{aligned}
\end{equation*}
First, taking the expectation in the second equation of (\ref{e221}) yields $\mathbb E\hat{p}=0$.
Applying It\^{o}'s formula to $\big<\hat{p},\hat{x}\big>$ and taking expectations on both sides:
\begin{equation*}
\begin{aligned}
    0&=\mathbb{E}\Big<G\big(\hat{x}_T-\mathbb{E}\hat{x}_T\big),\hat{x}_T\Big>\\
      &+\mathbb{E}\int_0^T\Big<(B^T\hat{p}_s+D^T\hat{q}_s), \widehat{\varphi}_s(\hat{p}, \hat{q})\Big> +\Big<\hat{x}_s, Q(\hat{x}_s-\mathbb{E}\hat{x}_s)\Big>ds+\mathbb{E}\int_0^T\big<\hat{p}_s, F\mathbb{E}\hat{x}_s\big>ds\\
      &\ge\mathbb{E}\Big<G^{\frac{1}{2}}\big(\hat{x}_T-\mathbb{E}\hat{x}_T\big),G^{\frac{1}{2}}\big(\hat{x}_T-\mathbb{E}\hat{x}_T\big)\Big>
      +\mathbb{E}\int_0^T\Big<Q^{\frac{1}{2}}(\hat{x}_s-\mathbb{E}\hat{x}_s), Q^{\frac{1}{2}}(\hat{x}_s-\mathbb{E}\hat{x}_s)\Big>ds.
\end{aligned}
\end{equation*}
Thus, we have, $G\big(\hat{x}_T-\mathbb{E}\hat{x}_T\big)=0$ and $Q\big(\hat{x}-\mathbb{E}\hat{x}\big)=0$
 which implies $\hat{p}_s\equiv 0,\hat{q}_s\equiv 0.$ Next, we have $\widehat{\varphi}_s(\hat{p}, \hat{q})\equiv0$ which further implies $\mathbb{E}\hat{x}_s \equiv 0,$ hence $\hat{x}_s \equiv 0.$ Hence the uniqueness follows.

\vskip 12pt

(\textbf{Existence}) Consider a family of parameterized FBSDE as follows:
\begin{equation*}
\left\{
    \begin{aligned}
      dx^{\alpha}&=\Big[\alpha {\mathbf{B}}(x, p^{\alpha}, q^{\alpha}, \mathbb{E}x^{\alpha})+\phi\Big]dt
      +\Big[\alpha \Xi(x, p^{\alpha}, q^{\alpha}, \mathbb{E}x^{\alpha})+\psi\Big]dW_t,\\
      -dp^{\alpha}&=\big[\alpha \mathbf{F}(x^{\alpha}, p^{\alpha}, \mathbb{E}x^{\alpha})+\gamma-\mathbb{E}\gamma\big]dt-q^{\alpha}dW_t, \\
      x^{\alpha}_0&=x,\quad \quad p^{\alpha}_{T}=-\alpha G\big(x^{\alpha}_{T}-\mathbb{E}x^{\alpha}_{T}\big)+\xi-\mathbb{E}\xi.
    \end{aligned}
    \right.
\end{equation*}
with
\begin{equation*}
\left\{
        \begin{aligned}
      {\mathbf{B}}&:= Ax+B\varphi(p,q)+F_1\mathbb{E}x+b, \\
     \Xi&:=D\varphi(p,q)+\sigma, \\
     \mathbf{F}&:=A^Tp-Q(x-\mathbb{E}x).
    \end{aligned}
    \right.
\end{equation*}
Here $(\phi, \psi, \gamma)$ are given process in $L_{\mathbb{F}^{W}}^{2}(0, T; \mathbb{R}^{n}) \times L_{\mathbb{F}^{W}}^{2}(0, T; \mathbb{R}^{n})\times L_{\mathbb{F}^{W}}^{2}(0, T; \mathbb{R}^{n}),$
and $\xi$ is a $\mathbb R^n$-valued square integrable random variable which is $\mathbb{F}_T^{W}$-measurable.   When $\alpha=0,$ we have a decoupled FBSDE whose solvability is trivial:
\begin{equation*}
\left\{
    \begin{aligned}
      dx&=\phi dt+\psi dW_t, \\
      -dp&=(\gamma-\mathbb{E}\gamma)dt-qdW_t, \\
      x_0&=x,\quad \quad p_{T}=\xi-\mathbb{E}\xi.
    \end{aligned}
    \right.
\end{equation*}
Denote $\mathcal{M}(0, T)=L_{\mathbb{F}^{W}}^{2}(0, T; \mathbb{R}^{n}) \times L_{\mathbb{F}^{W}}^{2, \mathcal{E}_0}(0, T; \mathbb{R}^{n})\times L_{\mathbb{F}^{W}}^{2}(0, T; \mathbb{R}^{n}).$ Now introduce a mapping ${I}_{\alpha_0}: (x, p, q) \in \mathcal{M}(0, T) \longrightarrow (X, P, Q) \in \mathcal{M}(0, T)$ via the following FBSDE:
\begin{equation*}
\left\{
    \begin{aligned}
      dX_t&=\Big[\alpha_0 \mathbf{B}(X_t, P_t, Q_t, \mathbb{E}X_t)+\delta \mathbf{B}(x_t, p_t, q_t, \mathbb{E}x_t)+\phi_t\Big]dt \\
        &+\Big[\alpha_0 \Xi(X_t, P_t, Q_t)+\delta \Xi(x_t, p_t, q_t)+\psi_t\Big]dW_t, \\
   -dP_t&=\Big[\alpha_0 \mathbf{F}(X_t, P_t, Q_t, \mathbb{E}X_t)+\gamma_t-\mathbb{E}\gamma+\delta\mathbf{F}(x_t, p_t, q_t, \mathbb{E}x_t)\Big]dt-Q_tdW_t, \\
      X_0&=x,\quad \quad P_{T}=-\alpha_0 G\big(X_T-\mathbb{E}X_T\big)-\delta G(x_{T}-\mathbb{E}x_T)+\xi-\mathbb{E}\xi.
    \end{aligned}
    \right.
\end{equation*}

Considering ${I}_{\alpha_0}: (x, p, q) \longrightarrow (X, P, Q)$ and ${I}_{\alpha_0}: (x', p', q') \longrightarrow (X', P', Q')$ and$$(\widehat{X}, \widehat{P}, \widehat{Q})=(X-X', P-P', Q-Q')$$
\begin{equation*}
\left\{
    \begin{aligned}
      d\widehat{X}_t&=\Big[\alpha_0 \widehat{\mathbf{B}}(\widehat{X}_t, \widehat{P}_t, \widehat{Q}_t, \mathbb{E}\widehat{X}_t)+\delta \widehat{\mathbf{B}}(\hat{x}_t, \hat{p}_t, \hat{q}_t, \mathbb{E}\hat{x}_t)\Big]dt\\
      &+\Big[\alpha_0 \widehat{\Xi}(\widehat{X}_t, \widehat{P}_t, \widehat{Q}_t)+\delta \widehat{\Xi}(\hat{x}_t, \hat{p}_t, \hat{q}_t)\Big]dW_t, \\
      -d\widehat{P}_t&=\Big[\alpha_0 \widehat{\mathbf{F}}(\widehat{X}_t, \widehat{P}_t, \widehat{Q}_t, \mathbb{E}\widehat{X}_t)+\delta(\widehat{\mathbf{F}}(\hat{x}_t, \hat{p}_t, \hat{q}_t, \mathbb{E}\hat{x}_t)\Big]dt-\widehat{Q}_tdW_t, \\
      \widehat{X}_0&=0,\quad \quad \widehat{P}_{T}=-\alpha_0 G\big(\widehat{X}_T-\mathbb{E}\widehat{X}_T\big)-\delta G(\hat{x}_{T}-\mathbb{E}\hat{x}_T),
    \end{aligned}
    \right.
\end{equation*}
with
\begin{equation*}
\left\{
        \begin{aligned}
     \widehat{\mathbf{B}}&:=\mathbf{B}(X_t, P_t, Q_t, \mathbb{E}X_t)-\mathbf{B}(X'_t, P'_t, Q'_t, \mathbb{E}X'_t), \\
     \widehat{\Xi}&:={\Xi}({X}_t, {P}_t, {Q}_t, \mathbb{E}{X}_t)-{\Xi}({X}'_t, {P}'_t, {Q}'_t, \mathbb{E}{X}'_t), \\
     \widehat{\mathbf{F}}&:=\mathbf{F}(X_t, P_t, Q_t, \mathbb{E}X_t)-\mathbf{F}(X'_t, P'_t, Q'_t, \mathbb{E}X'_t).
    \end{aligned}
    \right.
\end{equation*}

Note that $\mathbb{E}\widehat{P}_t\equiv 0$ because$$\widehat{\mathbf{F}}(X_t, P_t, Q_t, \mathbb{E}X_t)=A^{T}\widehat{P}-Q(\widehat{X}-\mathbb{E}\widehat{X}),  \quad \quad \text{and} \quad \mathbb{E}p_t \equiv 0.$$

Applying It\^{o} formula to $\big<\widehat{P},\widehat{X}\big>$ and taking expectations on both sides:
\begin{equation*}
\begin{aligned}
    &\mathbb{E}\Big<\widehat{X}_T, -\alpha_0 G\big(\widehat{X}_T-\mathbb{E}\widehat{X}_T\big)-\delta G(\hat{x}_{T}-\mathbb{E}\hat{x}_T)\Big>\\
      &=\mathbb{E}\int_0^T\Big<\widehat{X}_s, -\alpha_0 \widehat{\mathbf{F}}(\widehat{X}_s, \widehat{P}_s, \widehat{Q}_s, \mathbb{E}\widehat{X}_s)\Big>+\Big<\widehat{X}_s, -\delta\widehat{\mathbf{F}}(\hat{x}_s, \hat{p}_s, \hat{q}_s, \mathbb{E}\hat{x}_s)\Big> \\
      &+\Big<\widehat{P}_s, \alpha_0 \widehat{\mathbf{B}}(\widehat{X}_s, \widehat{P}_s, \widehat{Q}_s, \mathbb{E}\widehat{X}_s)\Big> +\Big<\widehat{P}_s, \delta\widehat{\mathbf{B}}(\hat{x}_s, \hat{p}_s, \hat{q}_s, \mathbb{E}\hat{x}_s)\Big>\\
      &+\Big<\widehat{Q}_s, \alpha_0 \widehat{\Xi}(\widehat{X}_s, \widehat{P}_s, \widehat{Q}_s)\Big>+\Big<\widehat{Q}_s, \delta\widehat{\mathbf{\Xi}}(\hat{x}_s, \hat{p}_s, \hat{q}_s)\Big>ds.\\
\end{aligned}
\end{equation*}

Rearranging the above terms, we have
\begin{equation*}
\begin{aligned}
    &\alpha_0\mathbb{E}\Big<\widehat{X}_T, G\big(\widehat{X}_T-\mathbb{E}\widehat{X}_T\big)\Big> +\mathbb{E}\int_0^T\alpha_0\left[\Big<\widehat{X}_s, -\widehat{\mathbf{F}}(\widehat{X}_s, \widehat{P}_s, \widehat{Q}_s, \mathbb{E}\widehat{X}_s)\Big>\right.\\
      &\left.+\Big<\widehat{P}_s, \widehat{\mathbf{B}}(\widehat{X}_s, \widehat{P}_s, \widehat{Q}_s, \mathbb{E}\widehat{X}_s)\Big>
      +\Big<\widehat{Q}_s, \widehat{\Xi}(\widehat{X}_s, \widehat{P}_s, \widehat{Q}_s)\Big>\right]ds \\
      &=\mathbb{E}\int_0^T\delta\left[\Big<\widehat{X}_s, \widehat{\mathbf{F}}(\hat{x}_s, \hat{p}_s, \hat{q}_s, \mathbb{E}\hat{x}_s)\Big>
      +\Big<\widehat{P}_s, -\widehat{\mathbf{B}}(\hat{x}_s, \hat{p}_s, \hat{q}_s, \mathbb{E}\hat{x}_s)\Big>\right. \\
      &\left.+\Big<\widehat{Q}_s, -\widehat{\Xi}(\hat{x}_s, \hat{p}_s, \hat{q}_s)\Big>\right]ds
      -\delta\mathbb{E}\Big<\widehat{X}_T, G(\hat{x}_{T}-\mathbb{E}\hat{x}_T)\Big>\\
\end{aligned}
\end{equation*}


Hence,
\begin{equation*}
\begin{aligned}
    &\alpha_0\mathbb{E} |G^{\frac{1}{2}}(\widehat{X}_T-\mathbb{E}\widehat{X}_T)|^{2}+\mathbb{E}\int_0^{T}\alpha_0|Q^{\frac{1}{2}}(\widehat{X}_s-\mathbb{E}\widehat{X}_s)|^{2}
    ds\\
    &\leq \alpha_0\mathbb{E}\Big<\widehat{X}_T, G\big(\widehat{X}_T-\mathbb{E}\widehat{X}_T\big)\Big> +\mathbb{E}\int_0^T\alpha_0\left[\Big<\widehat{X}_s, -\widehat{\mathbf{F}}(\widehat{X}_s, \widehat{P}_s, \widehat{Q}_s, \mathbb{E}\widehat{X}_s)\Big>\right.\\
      &\left.+\Big<\widehat{P}_s, \widehat{\mathbf{B}}(\widehat{X}_s, \widehat{P}_s, \widehat{Q}_s, \mathbb{E}\widehat{X}_s)\Big>
      +\Big<\widehat{Q}_s, \widehat{\Xi}(\widehat{X}_s, \widehat{P}_s, \widehat{Q}_s)\Big>\right]ds \\
      &=\mathbb{E}\int_0^T\delta\left[\Big<\widehat{X}_s, \widehat{\mathbf{F}}(\hat{x}_s, \hat{p}_s, \hat{q}_s, \mathbb{E}\hat{x}_s)\Big>
      +\Big<\widehat{P}_s, -\widehat{\mathbf{B}}(\hat{x}_s, \hat{p}_s, \hat{q}_s, \mathbb{E}\hat{x}_s)\Big>\right. \\
      &\left.+\Big<\widehat{Q}_s, -\widehat{\Xi}(\hat{x}_s, \hat{p}_s, \hat{q}_s)\Big>\right]ds
      -\delta\mathbb{E}\Big<\widehat{X}_T, G(\hat{x}_{T}-\mathbb{E}\hat{x}_T)\Big>\\
      &\leq \delta C_1 \mathbb{E}\int_0^{T} (|\hat{x}_s|^{2}+|\hat{p}_s|^{2}+|\hat{q}_s|^{2})ds+\delta C_1 \mathbb{E}\hat{x}_{T}^{2}+\delta C_1\mathbb{E}\int_0^{T} (|\widehat{X}_s|^{2}+|\widehat{P}_s|^{2}+|\widehat{Q}_s|^{2})ds+\delta C_1 \mathbb{E}\widehat{X}_{T}^{2}.
\end{aligned}
\end{equation*}

Then, by standard estimates of BSDE:
\begin{equation*}\begin{aligned}
    &\mathbb{E}\int_0^{T}\left(|\widehat{P}_s|^{2}+|\widehat{Q}_s|^{2}\right)ds\\& \leq \delta C_2 \mathbb{E}\int_0^{T} (|\hat{x}_s|^{2}+|\hat{p}_s|^{2}+
    |\hat{q}_s|^{2})ds+\delta C_2\mathbb E|\hat{x}_T|^2\\&+C_2\left(\alpha_0\mathbb{E} |G^{\frac{1}{2}}(\widehat{X}_T-\mathbb{E}\widehat{X}_T)|^{2}+\mathbb{E}\int_0^{T}\alpha_0|Q^{\frac{1}{2}}(\widehat{X}_s-\mathbb{E}\widehat{X}_s)|^{2}
    ds\right)\\
    &\le\delta C_3 \mathbb{E}\int_0^{T} (|\hat{x}_s|^{2}+|\hat{p}_s|^{2}+
    |\hat{q}_s|^{2})ds+\delta C_3\mathbb E|\hat{x}_T|^2.
\end{aligned}
\end{equation*}

Next, by the standard estimate of forward SDEs:
\begin{equation*}
\begin{aligned}
    &\mathbb{E}\int_0^{T}|\widehat{X}_s|^{2}ds+\mathbb{E}|\widehat{X}_T|^{2}\\& \leq \delta C_4 \mathbb{E}\int_0^{T} (|\hat{x}_s|^{2}+|\hat{p}_s|^{2}+|\hat{q}_s|^{2})ds+C_4\mathbb{E}\int_0^{T} \left(|\widehat{P}_s|^2+|\widehat{Q}_s|^{2}\right)ds\\& \leq \delta C_5\delta( \mathbb{E}\int_0^{T} (|\hat{x}_s|^{2}+|\hat{p}_s|^{2}+
    |\hat{q}_s|^{2})ds+\delta C_5\mathbb E|\hat{x}_T|^2
\end{aligned}
\end{equation*}
Based on the above estimates, we know the mapping $I$ satisfying$$\mathbb{E}\int_0^T \left(|\widehat{X}_s|^{2}+|\widehat{P}_s|^{2}+|\widehat{Q}_s|^{2}\right)ds+\mathbb{E}|\widehat{X}_T|^{2} \leq K\delta \left(\mathbb{E}\int_0^T \left(|\widehat{x}_s|^{2}+|\widehat{p}_s|^{2}+|\widehat{q}_s|^{2}\right)ds+\mathbb{E}|\widehat{x}_T|^{2}\right).$$It follows the mapping is a contraction and the existence follows immediately using the arguments presented in \cite{hp} and \cite{pw}.
\end{proof}


\section{$\epsilon$-Nash Equilibrium for Problem (\textbf{CC})}
In above sections, we can characterize the decentralized strategies $\{\bar{u}^i_t, 1\le i\le N\}$ of Problem (\textbf{CC})
through the auxiliary (\textbf{LCC}) and consistency condition system. For sake of presentation, we alter the notations of consistency condition system to be $(\alpha^{i}, \beta^{i}, \gamma^{i})$:
\begin{equation*}
\left\{
    \begin{aligned}
      d\alpha^{i}&=\Big[A\alpha^{i}+B\varphi(\beta^{i},\gamma^{i})+F\mathbb{E}\alpha^i+b\Big]dt + \Big[D\varphi(\beta^{i},\gamma^{i})+\sigma\Big]dW_i(t), \\
      d\beta^i&=-\big(A^T\beta^{i}-Q(\alpha^{i}-\mathbb{E}\alpha^{i})\big)dt+\gamma^idW_i(t), \\
      \alpha^{i}(0)&=x,\quad \quad \beta^i(T)=-G\big(\alpha^{i}(T)-\mathbb{E}\alpha^{i}(T)\big).
    \end{aligned}
    \right.
\end{equation*}
Now, we turn to verify the $\epsilon$-Nash equilibrium of them.
To start, we first present the definition of $\epsilon$-Nash equilibrium.
\begin{definition}\label{d1}
A set of strategies $\bar u_t^i\in \mathcal{U}^{c}_{ad}$,
$1\leq i\leq N,$ for $N$ agents is called to satisfy an
$\epsilon$-Nash equilibrium with respect to costs  $\mathcal J^i,\ 1\leq i\leq N,$
if there exists $\epsilon\geq0$ such that for any fixed $1\leq i\leq N$, we have
\begin{equation*}
\mathcal J^i(\bar{u}_t^i,\bar{u}_t^{-i})\leq \mathcal J^i(u_t^i,\bar{u}_t^{-i})+\epsilon,
\end{equation*}
when any alternative strategy $u^i\in \mathcal{U}^{c}_{ad}$ is applied by $\mathcal{A}_i$.
\end{definition}
\begin{remark}
If $\epsilon=0,$ then Definition \ref{d1} is reduced to the usual exact Nash equilibrium.
\end{remark}Now, we state the main result of this paper and its proof will be given later.
\begin{theorem}\label{Nash equilibrium theorem}
Under \emph{\textbf{(H1)-(H2)}}, $(\bar{u}_1,\bar{u}_2,\cdots,\bar{u}_N)$ is an $\epsilon$-Nash equilibrium of \emph{\textbf{Problem (CC)}}.
\end{theorem}
The proof of Theorem \ref{Nash equilibrium theorem} needs several lemmas which are presented later.
For agent $\mathcal{A}_i,$ recall that its decentralized open-loop optimal strategy is
$\bar{u}_i=\varphi(\beta^i,\gamma^i)$.
The decentralized state $\breve{x}_t^{i}$ is
\begin{equation}\label{decentralized state}
\left\{
        \begin{aligned}d\breve{x}^{i}&=\left[A\breve{x}^{i}\!+\!B\varphi(\beta^i,\gamma^i)
        \!+\!F\breve{x}^{(N)}\!+\!b\right]dt
        \!+\!\left[\!D\varphi(\beta^i,\gamma^i)
      +\!\sigma\right]dW_i(t),\\
      d\alpha^{i}&=\left[A\alpha^{i}\!+\!B\varphi(\beta^i,\gamma^i)\!+\!F\mathbb{E}\alpha^{i}\!+\!b\right]dt\!+\!
      \left[\!D\varphi(\beta^i,\gamma^i)+\!\sigma\right]dW_i(t),\\
      d\beta^{i}&=-\big[A^T\beta^{i}-Q(\alpha^{i}-\mathbb{E}\alpha^{i})\big]dt+\gamma^{i}dW_i(t),\\
      \breve{x}^{i}(0)&={\alpha}^{i}(0)=x,\qquad {\beta}^{i}(T)=-G\big({\alpha}^{i}(T)-\mathbb{E}{\alpha}^{i}(T)\big),
    \end{aligned}
    \right.
\end{equation}
where $\breve{x}^{(N)}=\frac{1}{N}\sum_{i=1}^{N}\breve{x}^{i}$.
For comparison, we prefer to write the limiting state $\alpha_t^{i}$ once again,
\begin{equation}\label{decentralized limiting state}
\left\{
  \begin{aligned}
d\alpha^{i}&=\left[A\alpha^{i}\!+\!B\varphi(\beta^i,\gamma^i)\!+\!F\mathbb{E}\alpha^{i}\!+\!b\right]dt\!+\!
      \left[\!D\varphi(\beta^i,\gamma^i)\!+\!\sigma\right]dW_i(t),\\
      d\beta^{i}&=-\big[A^T\beta^{i}-Q(\alpha^{i}-\mathbb{E}\alpha^{i})\big]dt+\gamma^{i}dW_i(t),\\
      \alpha^{i}(0)&=x,\qquad {\beta}^{i}(T)=-G\big({\alpha}^{i}(T)-\mathbb{E}{\alpha}^{i}(T)\big).
    \end{aligned}
    \right.
\end{equation}
Now, let us present the following lemmas.
\begin{lemma}\label{lemma for xN}
There exists a constant $C_0$ independent of $N$, such that
\begin{align*}
\sup_{1\leq i\leq N}\mathbb{E}\sup_{0\leq t\leq T}\Big|\breve{x}^{i}(t)\Big|^2\leq C_0.
\end{align*}
\end{lemma}
\begin{proof}
For each $1\leq i\leq N$, the monotonic fully coupled FBSDE
\[
\left\{
  \begin{aligned}
d\alpha^{i}&=\left[A\alpha^{i}\!+\!B\varphi(\beta^i,\gamma^i)\!+\!F\mathbb{E}\alpha^{i}\!+\!b\right]dt\!+\!
      \left[\!D\varphi(\beta^i,\gamma^i)\!+\!\sigma\right]dW_i(t),\\
      d\beta^{i}&=-\big[A^T\beta^{i}-Q(\alpha^{i}-\mathbb{E}\alpha^{i})\big]dt+\gamma^{i}dW_i(t),\\
      \alpha^{i}(0)&=x,\qquad {\beta}^{i}(T)=-G\big({\alpha}^{i}(T)-\mathbb{E}{\alpha}^{i}(T)\big),
    \end{aligned}
    \right.
\]
has a unique solution
$(\alpha^i,\beta^i,\gamma^i)\in L^{2}_{\mathbb{F}^i}(0,T;\mathbb{R}^n)\times L^{2}_{\mathbb{F}^i}(0,T;\mathbb{R}^n)
\times L^{2}_{\mathbb{F}^i}(0,T;\mathbb{R}^n)$. Thus, the system of all first equation of (\ref{decentralized state}), $1\leq i\leq N$,
has also a unique solution $(\breve{x}^i)_i\in (L^{2}_{\mathbb{F}^{W_1,\cdots,W_N}}(0,T;\mathbb{R}^n))^{\otimes N}$. Moreover, since $\{W_i\}_{i=1}^N$ is
$N$-dimensional Brownian motion whose components are independent and identically distributed, we have $(\alpha^i,\beta^i,\gamma^i), 1\leq i\leq N$
are independent identically distributed.

Noticing that $(\beta^i,\gamma^i)\in L^{2}_{\mathbb{F}^i}(0,T;\mathbb{R}^n)\times L^{2}_{\mathbb{F}^i}(0,T;\mathbb{R}^n)$,
with the Lipschitz property of the projection onto convex set,
it is not hard to show that $\varphi(\beta^i,\gamma^i):=\mathbf{P}_{\Gamma}\Big(R^{-1}
\big(B^T\beta^i+D^T\gamma^i\big)\Big)\in L^{2}_{\mathbb{F}^i}(0,T;\Gamma)$. From the above analysis and the classical estimates of FBSDEs, we have that there exists a constant $C_0$ independent of $N$ which may very line by line in the following, such that
\begin{equation*}
\mathbb{E}\sup_{0\leq t\leq T}(|\alpha^i(t)|^2+|\beta^i(t)|^2)
+\mathbb{E}\int_0^T(|\gamma^i(t)|^2+|\varphi(\beta^{i}(t),\gamma^{i}(t))|^2)dt\leq C_0.
\end{equation*}
Then from the first equation of (\ref{decentralized state}),  by using Burkholder-Davis-Gundy (BDG) inequality,
we have, for any $t\in[0,T]$,
\begin{equation}\label{ee8}
  \begin{aligned}
\mathbb{E}\sup_{0\leq s\leq t}|\breve{x}^{i}(s)|^2\leq & C_0+C_0\mathbb{E}\int_0^t\Big[|\breve{x}^{i}(s)|^2\!+\!|\breve{x}^{(N)}(s)|^2\Big]ds\\
\leq & C_0+C_0\mathbb{E}\int_0^t\Big[|\breve{x}^{i}(s)|^2\!+\!\frac{1}{N}\sum_{i=1}^N|\breve{x}^{i}(s)|^2\Big]ds.
\end{aligned}
\end{equation}
Thus
\[
\mathbb{E}\sup_{0\leq s\leq t}\sum_{i=1}^N|\breve{x}^{i}(s)|^2\leq \mathbb{E}\sum_{i=1}^N\sup_{0\leq s\leq t}|\breve{x}^{i}(s)|^2
\leq C_0N+2C_0\mathbb{E}\int_0^t\Big[\sum_{i=1}^N|\breve{x}^{i}(s)|^2\Big]ds.
\]
By Gronwall's inequality, it is easy to obtain
\[
\mathbb{E}\sup_{0\leq s\leq t}\sum_{i=1}^N|\breve{x}^{i}(s)|^2=O(N), \text{ for any } 1\leq i\leq N.
\]
Finally, substituting this estimate to (\ref{ee8}) and using Gronwall's inequality once again, we have
\[
\mathbb{E}\sup_{0\leq t\leq T}\Big|\breve{x}^{i}(t)\Big|^2\leq C_0, \text{ for any } 1\leq i\leq N,
\]
which completes the proof.
\end{proof}
\begin{lemma}\label{lemma for xN and m}
\begin{align}\label{estimate for xN and m}
\mathbb{E}\sup_{0\leq t\leq T}\Big|\breve{x}^{(N)}(t)-\mathbb{E}\alpha^{i}(t)\Big|^2=O\Big(\frac{1}{N}\Big).
\end{align}
\end{lemma}
\begin{proof}
On one hand, let us add up both sides of the first equation of (\ref{decentralized state})
with respect to all $1\leq i\leq N$ and multiply $\frac{1}{N}$, we obtain (recall that $\breve{x}^{(N)}=\frac{1}{N}\sum_{i=1}^{N}\breve{x}^{i}$)
\begin{equation}\label{ee1}
\left\{
\begin{aligned}
d\breve{x}^{(N)}&=\left[A\breve{x}^{(N)}\!+\!\frac{1}{N}\sum_{i=1}^{N}B\varphi(\beta^i,\gamma^i)\!+\!F\breve{x}^{(N)}\!+\!b\right]dt\\
&\qquad\!+\!\frac{1}{N}\sum_{i=1}^{N}\left[\!D\varphi(\beta^i,\gamma^i)
\!+\!\sigma\right]dW_i(t),\\
\breve{x}^{(N)}(0)&=x.
\end{aligned}
\right.
\end{equation}
On the other hand, by taking the expectation on both sides of the second equation of (\ref{decentralized state}),
it follows from Fubini's theorem that $\mathbb{E}\alpha^i$ satisfies the following equation:
\begin{equation}\label{ee2}
\left\{
\begin{aligned}
d(\mathbb{E}\alpha^{i})&=\left[A\mathbb{E}\alpha^{i}+\mathbb{E}\left(B\varphi(\beta^i,\gamma^i)\right)+F\mathbb{E}\alpha^{i}+b\right]dt,\\
\mathbb{E}{\alpha}^{i}(0)&=x.
\end{aligned}
\right.
\end{equation}
From (\ref{ee1}) and  (\ref{ee2}), by denoting  $\Delta(t):=\breve{x}^{(N)}(t)-\mathbb{E}\alpha^{i}(t)$, we have
\begin{equation*}
\left\{
\begin{aligned}
d\Delta&=\Bigg[A\Delta+\frac{1}{N}\sum_{i=1}^{N}B\varphi(\beta^i,\gamma^i)-B\mathbb{E}\varphi(\beta^i,\gamma^i)+F\Delta\Bigg]dt\\
&\quad+\frac{1}{N}\sum_{i=1}^{N}\Bigg[D\varphi(\beta^i,\gamma^i)
+\sigma\Bigg]dW_i(t),\\
\Delta(0)&=0,
\end{aligned}
\right.
\end{equation*}
and the inequality $(x+y)^2\leq 2x^2+2y^2$ yields that,  for any $t\in[0,T]$,
\[
\begin{aligned}
\mathbb{E}\sup_{0\leq s\leq t}|\Delta(s)|^2&\leq 2\mathbb{E}\sup_{0\leq s\leq t}\Bigg|\int_0^s\Big[(A\!+\!F)\Delta(r)
\!+\!\frac{1}{N}\sum_{i=1}^{N}B\varphi(\beta^i(r),\gamma^i(r))
\!-\!B\mathbb{E}\varphi(\beta^i(r),\gamma^i(r))\Big]dr\Bigg|^2\\
&+2\mathbb{E}\sup_{0\leq s\leq t}\Bigg|\frac{1}{N}\sum_{i=1}^{N}\int_0^s\Big[D\varphi(\beta^i(r),\gamma^i(r))+\sigma(r)\Big]dW_i(r)\Bigg|^2.
\end{aligned}
\]
From the Cauchy-Schwartz inequality and the BDG inequality, we obtain that there exists a constant $C_0$ independent of $N$ (which may vary line by line)
such that
\begin{equation}\label{ee3}
\begin{aligned}
&\mathbb{E}\sup_{0\leq s\leq t}|\Delta(t)|^2\leq C_0\mathbb{E}\int_0^t\Big[|\Delta(s)|^2
\!+\!\left|\frac{1}{N}\sum_{i=1}^{N}\varphi(\beta^i(s),\gamma^i(s))
\!-\!\mathbb{E}\varphi(\beta^i(s),\gamma^i(s))\right|^2\Big]ds\\
&\qquad+\frac{C_0}{N^2}\mathbb{E}\Bigg(\sum_{i=1}^{N}\int_0^t\left|\!D\varphi(\beta^i(s),\gamma^i(s))+\!\sigma(s)\right|^2ds\Bigg).
\end{aligned}
\end{equation}
Since $(\beta^i,\gamma^i), 1\leq i\leq N$ are independent identically distributed, for each fixed $s\in[0,T]$, let us denote that $\mu(s)=\mathbb{E}\varphi(\beta^i(s),\gamma^i(s))$ (note that $\mu$ does not depend on $i$), we have
\[
\begin{aligned}
&\mathbb{E}\left|\frac{1}{N}\sum_{i=1}^{N}\varphi(\beta^i(s),\gamma^i(s))-\mu(s)\right|^2
=\frac{1}{N^2}\mathbb{E}\left|\sum_{i=1}^{N}\left[\varphi(\beta^i(s),\gamma^i(s))-\mu(s)\right]\right|^2\\
=&\frac{1}{N^2}\mathbb{E}\sum_{i=1}^{N}\left|\varphi(\beta^i(s),\gamma^i(s))\!-\!\mu(s)\right|^2\\
&\!+\!\frac{1}{N^2}\mathbb{E}\sum_{ i=1,j=1,j\neq i,}^{N}\left\langle\varphi(\beta^i(s),\gamma^i(s))\!-\!\mu(s),\varphi(\beta^j(s),\gamma^j(s))\!-\!\mu(s)\right\rangle.
\end{aligned}
\]
Since $(\beta^i,\gamma^i), 1\leq i\leq N$ are independent, we have
\[
\begin{aligned}
&\frac{1}{N^2}\mathbb{E}\sum_{ i=1,j=1,j\neq i,}^{N}\left\langle\varphi(\beta^i(s),\gamma^i(s))\!-\!\mu(s),\varphi(\beta^j(s),\gamma^j(s))\!-\!\mu(s)\right\rangle\\
=&\frac{1}{N^2}\sum_{ i=1,j=1,j\neq i,}^{N}\left\langle\mathbb{E}\varphi(\beta^i(s),\gamma^i(s))\!-\!\mu(s),
\mathbb{E}\varphi(\beta^j(s),\gamma^j(s))\!-\!\mu(s)\right\rangle=0.
\end{aligned}
\]
Then, due to the fact that $(\beta^i,\gamma^i), 1\leq i\leq N$ are identically distributed, we can obtain that there exists a constant $C_0$ independent of $N$ such that
\[
\begin{aligned}
&\int_0^t\mathbb{E}\left|\frac{1}{N}\sum_{i=1}^{N}
B\varphi(\beta^i(s),\gamma^i(s))-B\mathbb{E}\varphi(\beta^i(s),\gamma^i(s))\right|^2ds\\
\leq & C_0\int_0^t\mathbb{E}\left|\frac{1}{N}\sum_{i=1}^{N}\varphi(\beta^i(s),\gamma^i(s))-\mu(s)\right|^2ds
= \frac{C_0}{N^2}\int_0^t\mathbb{E}\sum_{i=1}^{N}\left|\varphi(\beta^i(s),\gamma^i(s))\!-\!\mu(s)\right|^2ds\\
=& \frac{C_0}{N}\int_0^t\mathbb{E}\left|\varphi(\beta^i(s),\gamma^i(s))\!-\!\mu(s)\right|^2ds=O\Big(\frac{1}{N}\Big),
\end{aligned}
\]
where the last equality comes from the fact that $\varphi(\beta^i,\gamma^i)\in L^{2}_{\mathcal{F}^i}(0,T;\Gamma)$.

Let us now estimate the second term of (\ref{ee3}), using the fact that $(\alpha^i,\beta^i,\gamma^i)$ are  identically distributed, we have
$$
\frac{C_0}{N^2}\mathbb{E}\Bigg(\sum_{i=1}^{N}\int_0^t
\left|\!D\varphi(\beta^i(s),\gamma^i(s))\!+\!\sigma(s)\right|^2ds\Bigg)
=O\Big(\frac{1}{N}\Big).$$

Therefore, from the above analysis, we get from (\ref{ee3}) that
\[
\mathbb{E}\sup_{0\leq s\leq t}|\Delta(s)|^2\leq C_0\mathbb{E}\int_0^t|\Delta(s)|^2+O\Big(\frac{1}{N}\Big),  \text{ for any } t\in[0,T].
\]
Finally, by using Gronwall's inequality, we complete the proof.
\end{proof}

\begin{lemma}\label{lemma for x and xi}
\begin{align}\label{estimate for x and xi}
\sup_{1 \leq i \leq N}\mathbb{E}\sup_{0\leq t\leq T}\Big|\breve{x}^{i}(t)-\alpha^{i}(t)\Big|^2=O\Big(\frac{1}{N}\Big).
\end{align}
\end{lemma}
\begin{proof}
From (\ref{decentralized state}) and  (\ref{decentralized limiting state}), we have that
\begin{equation}\label{ee4}
\left\{
        \begin{aligned}d\breve{x}^{i}&=\left[A\breve{x}^{i}\!+\!B\varphi(\beta^i,\gamma^i)
        \!+\!F\breve{x}^{(N)}+b\right]dt
        \!+\!\left[D\varphi(\beta^i,\gamma^i)
      \!+\!\sigma\right]dW_i(t),\\
   d\alpha^{i}&=\left[A\alpha^{i}\!+\!B\varphi(\beta^i,\gamma^i)\!+\!F\mathbb{E}\alpha^{i}\!+\!b\right]dt
   \!+\!\left[\!D\varphi(\beta^i,\gamma^i)
      \!+\!\sigma\right]dW_i(t),\\
      \breve{x}^{i}(0)&=\alpha^{i}(0)=x,
    \end{aligned}
    \right.
\end{equation}
where $(\alpha^i,\beta^i,\gamma^i)$ is the unique solution to the following FBSDE:
\[
\left\{
 \begin{aligned}
   d\alpha^{i}&=\left[A\alpha^{i}\!+\!B\varphi(\beta^i,\gamma^i)\!+\!F\mathbb{E}\alpha^{i}\!+\!b\right]dt\!+\!
      \left[\!D\varphi(\beta^i,\gamma^i)\!+\!\sigma\right]dW_i(t),\\
      d\beta^{i}&=-\big[A^T\beta^{i}-Q(\alpha^{i}-\mathbb{E}\alpha^{i})\big]dt+\gamma^{i}dW_i(t),\\
   \alpha^{i}(0)&=x,\qquad \beta^{i}(T)=-G\big(\alpha^{i}_T-\mathbb{E}{\alpha}^{i}_T\big).
    \end{aligned}
    \right.
\]
From (\ref{ee4}), we have
\begin{equation*}\label{ee5}
\left\{
\begin{aligned}
d(\breve{x}^{i}-\alpha^{i})&=\Big[A(\breve{x}^{i}\!-\!\alpha^{i})
\!+\!F(\breve{x}^{(N)}\!-\!\mathbb{E}\alpha^{i})\Big]dt, \\
\breve{x}^{i}(0)-\bar{x}^{i}(0)&=0.
\end{aligned}
\right.
\end{equation*}
The classical estimate for the SDE yields that
\[
\mathbb{E}\sup_{0\leq t\leq T}\Big|\breve{x}^{i}(t)-\alpha^{i}(t)\Big|^2\leq
C_0\mathbb{E}\int_0^T\left|\breve{x}^{(N)}(s)-\mathbb{E}\alpha^{i}(s)\right|^2ds,
\]
where $C_0$ is a constant independent of $N$. Noticing (\ref{estimate for xN and m}) of Lemma \ref{lemma for xN and m},
we obtain (\ref{estimate for x and xi}). The proof is completed.
\end{proof}
\begin{lemma}\label{first lemma for cost}
For all $ 1\leq i\leq N$, we have
\begin{equation*}
 \Big|\mathcal{J}_i(\bar{u}_i, \bar{u}_{-i})-J_i(\bar{u}_i)\Big|=O\Big(\frac{1}{\sqrt{N}}\Big).
\end{equation*}
\end{lemma}
\begin{proof}
Recall (\ref{original cost}), (\ref{limit cost}) and (\ref{limit average process}), we have
\[
\begin{aligned}
    \mathcal {J}_i(\bar{u}_i,\bar{u}_{-i})&=\frac{1}{2}\mathbb{E}
    \bigg [ \int_0^T \Big <Q(t)\big(\breve{x}^i(t)-\breve{x}^{(N)}(t)\big),\breve{x}^i(t)-\breve{x}^{(N)}(t)\Big>\\
    &+\big<R(t)\bar{u}_i(t),\bar{u}_i(t)\big>dt+\Big<G\big(\breve{x}^i(T)-\breve{x}^{(N)}(T)\big),\breve{x}^i(T)-\breve{x}^{(N)}(T)\Big>\bigg]
\end{aligned}
\]
and
\[
\begin{aligned}
    {J}_i(\bar{u}_i)&=\frac{1}{2}\mathbb{E}
    \bigg [ \int_0^T \Big <Q(t)\big(\alpha^i(t)-\mathbb{E}\alpha^i(t)\big),\alpha^i(t)-\mathbb{E}\alpha^i(t)\Big>dt\\
    &+\big<R(t)\bar{u}_i(t),\bar{u}_i(t)\big>dt+\Big<G\big(\alpha^i(T)-\mathbb{E}\alpha^i(T)\big),\alpha^i(T)-\mathbb{E}\alpha^i(T)\Big>\bigg],
\end{aligned}
\]
then
\begin{equation}\label{ee7}
\begin{aligned}
&\mathcal{J}_i(\bar{u}_i, \bar{u}_{-i})-J_i(\bar{u}_i)\\
=&\frac{1}{2}\mathbb{E}\bigg[\int_0^T\left( \Big <Q(t)\big(\breve{x}^i(t)\!-\!\breve{x}^{(N)}(t)\big),\breve{x}^i(t)\!-\!\breve{x}^{(N)}(t)\Big>
\!-\!\Big <Q(t)\big(\alpha^i(t)\!-\!\mathbb{E}\alpha^i(t)\big),\alpha^i(t)\!-\!\mathbb{E}\alpha^i(t)\Big>\right)dt\\
&\!+\!\Big<G\big(\breve{x}^i(T)\!-\!\breve{x}^{(N)}(T)\big),\breve{x}^i(T)\!-\!\breve{x}^{(N)}(T)\Big>
\!-\!\Big<G\big(\alpha^i(T)\!-\!\mathbb{E}\alpha^i(T)\big),\alpha^i(T)\!-\!\mathbb{E}\alpha^i(T)\Big>\bigg].
\end{aligned}
\end{equation}
From
\[
\begin{aligned}
&\langle Q(a-b), a-b\rangle-\langle Q(c-d), c-d\rangle\\
=&\langle Q(a-b-(c-d)), a-b-(c-d)\rangle+2\langle Q(a-b-(c-d)), c-d\rangle,
\end{aligned}
\]
and Lemma \ref{lemma for xN and m}, Lemma \ref{lemma for x and xi} as well as $\mathbb{E}\sup_{0\leq t\leq T}\left|\alpha^i(t)\right|^2\leq C_0$,
for some constant $C_0$ independent of $N$ which may vary line by line in the following,  we have
\[
\begin{aligned}
&\left|\mathbb{E}\bigg [\int_0^T\left( \Big< Q(t)\big(\breve{x}^i(t)\!-\!\breve{x}^{(N)}(t)\big),\breve{x}^i(t)\!-\!\breve{x}^{(N)}(t)\Big>
\!-\!\Big< Q(t)\big(\alpha^i(t)\!-\!\mathbb{E}\alpha^i(t)\big),\alpha^i(t)\!-\!\mathbb{E}\alpha^i(t)\Big>\right)dt\right|\\
\leq& C_0\int_0^T\mathbb{E}\left|\breve{x}^i(t)-\breve{x}^{(N)}(t)-(\alpha^i(t)-\mathbb{E}\alpha^i(t))\right|^2dt\\
&\qquad\qquad+C_0\int_0^T\mathbb{E}\left[
\left|\breve{x}^i(t)-\breve{x}^{(N)}(t)-(\alpha^i(t)-\mathbb{E}\alpha^i(t))\right|\cdot
\left|\alpha^i(t)-\mathbb{E}\alpha^i(t)\right|\right]dt\\
\leq &C_0\int_0^T\mathbb{E}\left|\breve{x}^i(t)-\alpha^i(t)\right|^2dt
+C_0\int_0^T\mathbb{E}\left|\breve{x}^{(N)}(t)-\mathbb{E}\alpha^i(t)\right|^2dt\\
&\qquad\qquad+C_0\int_0^T\left(\mathbb{E}\left|\breve{x}^i(t)-\breve{x}^{(N)}(t)
-(\alpha^i(t)-\mathbb{E}\alpha^i(t))\right|^2\right)^{\frac{1}{2}}
\left(\mathbb{E}\left|\alpha^i(t)-\mathbb{E}\alpha^i(t)\right|^2\right)^{\frac{1}{2}}dt\\
\leq &C_0\int_0^T\mathbb{E}\left|\breve{x}^i(t)-\alpha^i(t)\right|^2dt
+C_0\int_0^T\mathbb{E}\left|\breve{x}^{(N)}(t)-\mathbb{E}\alpha^i(t)\right|^2dt\\
&\qquad\qquad+C_0\int_0^T\left(\mathbb{E}\left|\breve{x}^i(t)-\alpha^i(t)\right|^2
+\mathbb{E}\left|\breve{x}^{(N)}(t)-\mathbb{E}\alpha^i(t)\right|^2\right)^{\frac{1}{2}}
dt\\
=& O\left(\frac{1}{\sqrt{N}}\right).
\end{aligned}
\]
With similar argument, we can show that
\[
\begin{aligned}
\left|\mathbb{E}\bigg [\Big< G\big(\breve{x}^i(T)\!-\!\breve{x}^{(N)}(T)\big),\breve{x}^i(T)\!-\!\breve{x}^{(N)}(T)\Big>
\!-\!\Big< G\big(\alpha^i(T)\!-\!\mathbb{E}\alpha^i(T)\big),\alpha^i(T)\!-\!\mathbb{E}\alpha^i(T)\Big>\bigg]\right|
=O\left(\frac{1}{\sqrt{N}}\right).
\end{aligned}
\]
The proof is completed by noticing (\ref{ee7}).
\end{proof}

\medskip

Our remaining analysis is to prove the control strategies set $(\bar{u}_1,\bar{u}_2,\ldots,\bar{u}_N)$ is an $\epsilon$-Nash equilibrium for \textbf{ Problem (CC)}. For any fixed $i$, $1\leq i \leq N$,
we consider the perturbation control $u_i \in \mathcal{U}_{ad}^{d,i}$ and we have the following state dynamics ($j\neq i$):
\begin{equation}\label{ee9}
\left\{
        \begin{aligned}dy^{i}&=\left[Ay^{i}+Bu_i+Fy^{(N)}+b\right]dt
        +\left[Du_{i}+\sigma\right]dW_i(t),\\
      dy^{j}&=\left[Ay^{j}+B\varphi(\beta^j,\gamma^j)
      +Fy^{(N)}+b\right]dt+\left[D\varphi(\beta^j,\gamma^j)
      +\sigma\right]dW_j(t),\\
      d\alpha^{j}&=\left[A\alpha^{j}+B\varphi(\beta^j,\gamma^j)
      +F\mathbb{E}\alpha^{j}+b\right]dt+\left[D\varphi(\beta^j,\gamma^j)+\sigma \right]dW_j(t),\\
      d\beta^{j}&=-\big[A^T\beta^{j}-Q(\alpha^{j}-\mathbb{E}\alpha^{j})\big]dt+\gamma^{j}dW_j(t),\\
       y^{i}(0)&=y^{j}(0)=\alpha^{j}(0)=x,\qquad \beta^{j}(T)=-G\big(\alpha^{j}(T)-\mathbb{E}\alpha^{j}(T)\big),
    \end{aligned}
    \right.
\end{equation}
where $y^{(N)}=\frac{1}{N}\sum_{i=1}^{N}y^{i}$. The wellposedness of above system is easily to obtain.
To prove $(\bar{u}_1,\bar{u}_2,\ldots,\bar{u}_N)$ is an $\epsilon$-Nash equilibrium, we need to show that for $ 1\leq i\leq N$,
\[
\inf_{u_i\in\mathcal{U}_{ad}^i}\mathcal{J}_i(u_i,\bar{u}_{-i})\ge\mathcal{J}_i(\bar{u}_i,\bar{u}_{-i})-\epsilon.
\]
Then we only need to consider the perturbation $u_i\in\mathcal{U}_{ad}^{d,i}$
such that $\mathcal{J}_i(u_i,\bar{u}_{-i})\leq\mathcal{J}_i(\bar{u}_i,\bar{u}_{-i})$. Thus we have
\[
\mathbb{E}\int_0^T\langle Ru_i(t),u_i(t)\rangle dt\leq
\mathcal{J}_i(u_i,\bar{u}_{-i})\leq\mathcal{J}_i(\bar{u}_i,\bar{u}_{-i})\leq J_i(\bar{u}_i)+O\Big(\frac{1}{\sqrt{N}}\Big),
\]
which implies that
\begin{equation}\label{boundness of control}
\mathbb{E}\int_0^T|u_i(t)|^2dt\leq C_0,
\end{equation}
where $C_0$ is a constant independent of $N$. Then similar to Lemma \ref{lemma for xN},
we can show that there exists a constant $C_0$ independent of $N$ such that
\begin{equation}\label{boundedness of yi}
\sup_{1\leq i\leq N}\mathbb{E}\sup_{0\leq t\leq T}|y^i(t)|^2\leq C_0.
\end{equation}

\medskip

Now, for the $i^{th}$ agent, we consider the perturbation in the \textbf{Problem (LCC)}.
We introduce the following system of the decentralized limiting state  with perturbation control  ($j\neq i$):
\begin{equation}\label{ee10}
\left\{
\begin{aligned}
d\bar{y}^{i}&=\left[A\bar{y}^{i}+Bu_i+F\mathbb{E}\alpha^{i}+b\right]dt
+\left[Du_i+\sigma\right]dW_i(t),\\
      d\alpha^{j}&=\left[A\alpha^{j}+B\varphi(\beta^j,\gamma^j)
      +F\mathbb{E}\alpha^{j}+b\right]dt
      +\left[D\varphi(\beta^j,\gamma^j)+\sigma\right]dW_j(t),\\
      d\beta^{j}&=-\big[A^T\beta^{j}-Q(\alpha^{j}-\mathbb{E}\alpha^{j})\big]dt+\gamma^{j}dW_j(t),\\
      \bar{y}^{i}(0)&=\alpha^{j}(0)=x,\qquad \beta^{j}(T)=-G\big(\alpha^{j}(T)-\mathbb{E}\alpha^{j}(T)\big).
    \end{aligned}
    \right.
\end{equation}
We have the following results:
\begin{lemma}\label{lemma 2}
\begin{align}\label{lemma 2 estimate}
\mathbb{E}\sup_{0\leq t\leq T}\Big|y^{(N)}(t)-\mathbb{E}\alpha^{i}(t)\Big|^2=O\Big(\frac{1}{N}\Big).
\end{align}
\end{lemma}
\begin{proof}
By (\ref{ee9}), we get
\begin{equation}\label{ee12}
\left\{
        \begin{aligned}dy^{(N)}&=\left[(A+F)y^{(N)}\!+\!\frac{1}{N}Bu_i\!+\!\frac{1}{N}\sum\limits_{j=1,j\neq i}^{N}
        B\varphi(\beta^j,\gamma^j)\!+\!b\right]dt \\
        &\qquad\!+\!\frac{1}{N}\sum\limits_{j=1}^{N}\sigma dW_j(t)+\!\frac{1}{N}Du_idW_i(t)\!+\!\frac{1}{N}\sum\limits_{j=1,j\neq i}^{N}D\varphi(\beta^j,\gamma^j)dW_j(t), \\
       y^{(N)}(0)&=x.
    \end{aligned}
    \right.
\end{equation}
Let us denote $\Pi:=y^{(N)}-\mathbb{E}\alpha^i$, and recall (\ref{ee2}) which is
\[
\left\{
\begin{aligned}
d(\mathbb{E}\alpha^{i})&=\left[A\mathbb{E}\alpha^{i}+\mathbb{E}B\varphi(\beta^i,\gamma^i)
+F\mathbb{E}\alpha^{i}+b\right]dt,\\
\mathbb{E}\alpha^{i}(0)&=x,
\end{aligned}
\right.
\]
we have
\begin{equation*}
\left\{
\begin{aligned}
d\Pi=&\Bigg[(A+F)\Pi\!+\!\frac{1}{N}Bu_i\!+\!\left(\frac{1}{N}\sum\limits_{j=1,j\neq i}^{N}
        B\varphi(\beta^j,\gamma^j)-\mathbb{E}B\varphi(\beta^j,\gamma^j)‰\right)\Bigg]dt\\
        &\frac{1}{N}\sum\limits_{j=1}^{N}\sigma dW_j(t)
         +\frac{1}{N}Du_idW_i(t)\!+\!\frac{1}{N}\sum\limits_{j=1,j\neq i}^{N}D\varphi(\beta^j,\gamma^j)dW_j(t),\\
\Pi(0)&=0.
\end{aligned}
\right.
\end{equation*}
By the Cauchy-Schwartz inequality as well as the BDG inequality,
we obtain that there exists a constant $C_0$ independent of $N$ which may vary line by line such that, for any $t\in[0,T]$,
\begin{equation}\label{ee11}
\begin{aligned}
\mathbb{E}\sup_{0\leq s\leq t}|\Pi(s)|^2&\leq C_0\mathbb{E}\int_0^t\left(|\Pi(s)|^2\!+\!\frac{1}{N^2}|u_i(s)|^2\right)ds\\
&\!+\!C_0\mathbb{E}\int_0^t\left|\frac{1}{N}\sum\limits_{j=1,j\neq i}^{N}\varphi(\beta^j(s),\gamma^j(s))
\!-\!\mathbb{E}\varphi(\beta^j(s),\gamma^j(s))\right|^2ds\\
&\!+\!\frac{C_0}{N^2}\mathbb{E}\sum_{j=1}^{N}\int_0^t\left|\sigma(s)\right|^2ds\\
&\!+\!\frac{C_0}{N^2}\mathbb{E}\int_0^t|u_i(s)|^2ds\!+\!\frac{C_0}{N^2}\mathbb{E}\sum\limits_{j=1,j\neq i}^{N}\int_0^t|\varphi(\beta^j(s),\gamma^j(s)|^2ds.
\end{aligned}
\end{equation}
On the one hand, by denoting $\mu(s):=\mathbb{E}\varphi(\beta^j(s),\gamma^j(s))$ (note that since $(\alpha^j,\beta^j,\gamma^j)$, $1\leq j\leq N$, $j\neq i$, are independent identically distributed, thus $\mu$ is independent of $j$), we have
\[
\begin{aligned}
&\mathbb{E}\left|\frac{1}{N}\sum\limits_{j=1,j\neq i}^{N}\varphi(\beta^j(s),\gamma^j(s))
-\mu(s)\right|^2\\
\leq&2\mathbb{E}\left|\frac{1}{N}\sum\limits_{j=1,j\neq i}^{N}\varphi(\beta^j(s),\gamma^j(s))
-\frac{N-1}{N}\mu(s)\right|^2
+2\mathbb{E}\left|\frac{1}{N}\mu(s)\right|^2\\
=&2\frac{(N-1)^2}{N^2}\mathbb{E}\left|\frac{1}{N-1}\sum\limits_{j=1,j\neq i}^{N}\varphi(\beta^j(s),\gamma^j(s))
-\mu(s)\right|^2+\frac{2}{N^2}\mathbb{E}|\mu(s)|^2.
\end{aligned}
\]
Then, due to the fact that $(\beta^i,\gamma^i), 1\leq i\leq N$ are identically distributed and  $\varphi(\beta^i,\gamma^i)\in L^{2}_{\mathbb{F}^i}(0,T;\Gamma)$,  similarly to Lemma \ref{lemma for xN and m} we can obtain that there exists a constant $C_0$ independent of $N$ such that
\[
\begin{aligned}
&\int_0^t\mathbb{E}\left|\frac{1}{N}\sum\limits_{j=1,j\neq i}^{N}\varphi(\beta^j(s),\gamma^j(s))
-\mathbb{E}\varphi(\beta^j(s),\gamma^j(s))\right|^2\\
\leq & \frac{C_0(N-1)^2}{N^2}\int_0^t\mathbb{E}\left|\frac{1}{N-1}\sum\limits_{j=1,j\neq i}^{N}\varphi(\beta^j(s),\gamma^j(s))-\mu(s)\right|^2ds
+\frac{C_0}{N^2}\int_0^t\mathbb{E}|\mu(s)|^2ds\\
=& \frac{C_0(N-1)}{N^2}\int_0^t\mathbb{E}\left|\varphi(\beta^j(s),\gamma^j(s))
\!-\!\mu(s)\right|^2ds+\frac{C_0}{N^2}\int_0^t\mathbb{E}|\mu(s)|^2ds\\
=&O\Big(\frac{1}{N}\Big).
\end{aligned}
\]
In addition, due to (\ref{boundness of control}) and (\ref{boundedness of yi}), we get
\[
\begin{aligned}
\frac{C_0}{N^2}\mathbb{E}\int_0^t|u_i(s)|^2ds
+\frac{C_0}{N^2}\mathbb{E}\sum_{j=1}^{N}\int_0^t\left|\sigma(s)\right|^2ds
= O\Big(\frac{1}{N}\Big).
\end{aligned}
\]
and similarly, since  $(\beta^j,\gamma^j)$, $1\leq j\leq N$, $j\neq i$, are identically distributed, we have
\[
\frac{C_0}{N^2}\mathbb{E}\sum\limits_{j=1,j\neq i}^{N}\int_0^t|\varphi(\beta^j(s),\gamma^j(s)|^2ds=O\Big(\frac{1}{N}\Big).
\]
Therefore, from above estimates, we get from (\ref{ee11}) that, for any $t\in[0,T]$,
\[
\mathbb{E}\sup_{0\leq s\leq T}|\Pi(s)|^2\leq C_0\mathbb{E}\int_0^t|\Pi(s)|^2ds+O\Big(\frac{1}{N}\Big).
\]
Finally, by using Gronwall's inequality, we complete the proof.
\end{proof}
\begin{lemma}\label{lemma 3}
\begin{align}\label{lemma 3 estimate}
\mathbb{E}\sup_{0\leq t\leq T}\Big|y^{i}_t-\bar{y}^{i}_t\Big|^2=O\Big(\frac{1}{N}\Big).
\end{align}
\end{lemma}
\begin{proof}
From respectively the first equation of (\ref{ee9}) and (\ref{ee10}), we obtain
\[
\left\{
\begin{aligned}
d(y^i-\bar{y}^i)&=\left[A(y^i-\bar{y}^i)+F(y^{(N)}-\mathbb{E}\alpha^i)\right]dt
,\\
y^i(0)-\bar{y}^i(0)&=0.
\end{aligned}
\right.
\]
With the help of classical estimates of SDE,  Gronwall's inequality and  (\ref{lemma 2 estimate}) of Lemma \ref{lemma 2},
it is easily to obtain (\ref{lemma 3 estimate}). The proof is completed.
\end{proof}

\begin{lemma}\label{lemma 4}
For all $ 1\leq i\leq N,$ for the perturbation control $u_{i}$, we have
\begin{equation*}
\Big|\mathcal{J}_i({u}_i, \bar{u}_{-i})-J_i({u}_i)\Big|=O\Big(\frac{1}{\sqrt{N}}\Big).
\end{equation*}
\end{lemma}
\begin{proof}
Recall (\ref{original cost}), (\ref{limit cost}) and (\ref{limit average process}), we have
\begin{equation*}
\begin{aligned}
&\mathcal{J}_i({u}_i, \bar{u}_{-i})-J_i({u}_i)\\
=&\frac{1}{2}\mathbb{E}\bigg [ \int_0^T\left( \Big <Q(t)\big(y^i(t)-y^{(N)}(t)\big),y^i(t)-y^{(N)}(t)\Big>-\Big <Q(t)\big(\bar{y}^i(t)
-\mathbb{E}\alpha^i(t)\big),\bar{y}^i(t)-\mathbb{E}\alpha^i(t)\Big>\right)dt\\
&+\Big<G\big(y^i(T)-y^{(N)}(T)\big),y^i(T)-y^{(N)}(T)\Big>-\Big<G\big(\bar{y}^i(T)
-\mathbb{E}\alpha^i(T)\big),\bar{y}^i(T)-\mathbb{E}\alpha^i(T)\Big>\bigg].
\end{aligned}
\end{equation*}
Using Lemma \ref{lemma 2} and Lemma \ref{lemma 3} as well as $\mathbb{E}\sup_{0\leq t\leq T}\left(|\bar{y}^i(t)|^2+|\alpha^i(t)|^2\right)\leq C_0$,
 for some constant $C_0$ independent of $N$ which may vary line by line in the following,  we have
\[
\begin{aligned}
&\left|\mathbb{E}\bigg [ \int_0^T\left( \Big <Q(t)\big(y^i(t)-y^{(N)}(t)\big),y^i(t)-y^{(N)}(t)\Big>
-\Big <Q(t)\big(\bar{y}^i(t)-\mathbb{E}\alpha^i(t)\big),\bar{y}^i(t)-\mathbb{E}\alpha^i(t)\Big>\right)dt\right|\\
\leq& C_0\int_0^T\mathbb{E}\left|y^i(t)-y^{(N)}(t)-(\bar{y}^i(t)-\mathbb{E}\alpha^i(t))\right|^2dt\\
&\qquad\qquad\qquad+C_0\int_0^T\mathbb{E}\left[\left|y^i(t)-y^{(N)}(t)-(\bar{y}^i(t)-\mathbb{E}\alpha^i(t))\right|\cdot
\left|\bar{y}^i(t)-\mathbb{E}\alpha^i(t)\right|\right]dt\\
\leq &C_0\int_0^T\mathbb{E}\left|y^i(t)-\bar{y}^i(t)\right|^2dt
+C_0\int_0^T\mathbb{E}\left|y^{(N)}(t)-\mathbb{E}\alpha^i(t)\right|^2dt\\
&\qquad\qquad\qquad+C_0\int_0^T\left(\mathbb{E}\left|y^i(t)-y^{(N)}(t)-
(\bar{y}^i(t)-\mathbb{E}\alpha^i(t))\right|^2\right)^{\frac{1}{2}}
\left(\mathbb{E}\left|\bar{y}^i(t)-\mathbb{E}\alpha^i(t)\right|^2\right)^{\frac{1}{2}}dt\\
\leq &C_0\int_0^T\mathbb{E}\left|y^i(t)-\bar{y}^i(t)\right|^2dt
+C_0\int_0^T\mathbb{E}\left|y^{(N)}(t)-\mathbb{E}\alpha^i(t)\right|^2dt\\
&\qquad\qquad\qquad+C_0\int_0^T\left(\mathbb{E}\left|y^i(t)-\bar{y}^i(t)\right|^2
+\mathbb{E}\left|y^{(N)}(t)-\mathbb{E}\alpha^i(t)\right|^2\right)^{\frac{1}{2}}
dt\\
=& O\left(\frac{1}{\sqrt{N}}\right).
\end{aligned}
\]
With similar argument, we can show that
\[
\begin{aligned}
&\left|\mathbb{E}\bigg [\Big<G\big(y^i(T)-y^{(N)}(T)\big),y^i(T)-y^{(N)}(T)\Big>-\Big<G\big(\bar{y}^i(T)
-\mathbb{E}\alpha^i(T)\big),\bar{y}^i(T)-\mathbb{E}\alpha^i(T)\Big>\bigg]\right|\\
=&O\left(\frac{1}{\sqrt{N}}\right).
\end{aligned}
\]
The proof is completed by noticing (\ref{ee12}).
\end{proof}
\medskip

\noindent\textbf{Proof of Theorem \ref{Nash equilibrium theorem}:} Now, we consider the $\epsilon$-Nash equilibrium for $\mathcal{A}_i$ for
\textbf{Problem (CC)}. Combining Lemma \ref{first lemma for cost} and Lemma \ref{lemma 4}, we have
\[
\mathcal{J}_i(\bar{u}_i, \bar{u}_{-i})=J_i(\bar{u}_i)+O\Big(\frac{1}{\sqrt{N}}\Big)\leq J_i({u}_i)+O\Big(\frac{1}{\sqrt{N}}\Big)
=\mathcal{J}_i({u}_i, \bar{u}_{-i})+O\Big(\frac{1}{\sqrt{N}}\Big).
\]
Consequently, Theorem \ref{Nash equilibrium theorem} holds with $\epsilon=O(\frac{1}{\sqrt{N}})$.\hfill$\Box$

\section{Appendix}

For the readers' convenient, let us recall the following properties of  projection $\mathbf{P}_{\Gamma}$ onto a closed convex set, see \cite{Brezis 2010}, Chapter 5.
\begin{theorem}\label{projection theorem}
For a nonempty  closed  convex set $\Gamma\subset\mathbb{R}^m$, for every $x\in\mathbb{R}^m$, there exists a unique $x^\ast\in\Gamma$, such that
\[
|x-x^\ast|=\min_{y\in\Gamma}|x-y|=:dist(x,\Gamma).
\]
Moreover, $x^\ast$ is characterized by the property
\begin{equation}\label{projection characterization}
x^\ast\in\Gamma, \quad \big<x^\ast-x, x^\ast-y\big>\leq0 \qquad \forall y \in \Gamma.
\end{equation}
The above element $x^\ast$ is called the projection of $x$ onto $\Gamma$ and is denoted by $\mathbf{P}_{\Gamma}[x]$.
\end{theorem}

From above theorem, it is easy to show that
\begin{proposition}
Let $\Gamma\subset\mathbb{R}^m$ be a nonempty  closed  convex set, then we have
 \begin{equation}\label{projection inequality}
\big|\mathbf{P}_{\Gamma}[x]-\mathbf{P}_{\Gamma}[y]\big|^2\leq    \big<\mathbf{P}_{\Gamma}[x]-\mathbf{P}_{\Gamma}[y],x-y\big>.
\end{equation}
\end{proposition}
\begin{proof} From (\ref{projection characterization}), we have
\begin{equation}\label{projection 1}
\big<\mathbf{P}_{\Gamma}[x]-x, \mathbf{P}_{\Gamma}[x]-z\big>\leq0 \qquad \forall z \in \Gamma.
\end{equation}
and
\begin{equation}\label{projection 2}
\big<\mathbf{P}_{\Gamma}[y]-y, \mathbf{P}_{\Gamma}[y]-z\big>\leq0 \qquad \forall z \in \Gamma.
\end{equation}
Choosing $z=\mathbf{P}_{\Gamma}[y]$ in (\ref{projection 1}) and $z=\mathbf{P}_{\Gamma}[x]$ in (\ref{projection 2}), then adding the corresponding inequalities, we obtain
\[
\big<\mathbf{P}_{\Gamma}[x]-x, \mathbf{P}_{\Gamma}[x]-\mathbf{P}_{\Gamma}[y]\big>+\big<\mathbf{P}_{\Gamma}[y]-y, \mathbf{P}_{\Gamma}[y]-\mathbf{P}_{\Gamma}[x]\big>\leq0,
\]
which yields obviously
\[
\big|\mathbf{P}_{\Gamma}[x]-\mathbf{P}_{\Gamma}[y]\big|^2\leq    \big<\mathbf{P}_{\Gamma}[x]-\mathbf{P}_{\Gamma}[y],x-y\big>.
\]\hfill
\end{proof}
\begin{proposition}
Let $\Gamma\subset\mathbb{R}^m$ be a nonempty  closed  convex set, then the projection $\mathbf{P}_{\Gamma}$ does not increase the distance, i.e.
 \begin{equation}\label{projection lipschitz}
    \big|\mathbf{P}_{\Gamma}[x]-\mathbf{P}_{\Gamma}[y]\big| \leq \big|x-y\big|.
\end{equation}
\end{proposition}
\begin{proof}
From  (\ref{projection inequality}), we have
\[
\big|\mathbf{P}_{\Gamma}[x]-\mathbf{P}_{\Gamma}[y]\big|^2\leq    \big<\mathbf{P}_{\Gamma}[x]-\mathbf{P}_{\Gamma}[y],x-y\big>\leq \left|\mathbf{P}_{\Gamma}[x]-\mathbf{P}_{\Gamma}[y]\right|\cdot|x-y|,
\]
which gives directly (\ref{projection lipschitz}).
\end{proof}

Now let us consider $\mathbb{R}^m$ and the projection $\mathbf{P}_{\Gamma}$ both with the norm $\|\cdot\|_{R_0}:=\langle R_0^{\frac{1}{2}}\cdot, R_0^{\frac{1}{2}}\cdot\rangle$, from (\ref{projection inequality}), we have
\begin{proposition}Let $\Gamma\subset\mathbb{R}^m$ be a nonempty  closed  convex set, then
\[
\langle\langle\mathbf{P}_{\Gamma}[x]
   -\mathbf{P}_{\Gamma}[y],x-y\rangle\rangle=\left\langle R^{\frac{1}{2}}\bigg(\mathbf{P}_{\Gamma}[x]
   -\mathbf{P}_{\Gamma}[y]\bigg),R^{\frac{1}{2}}(x-y)\right\rangle \ge 0.
\]
\end{proposition}

\end{document}